\documentclass[12pt]{amsart}

\usepackage{amssymb, amsmath, amsthm, mathtools,}
\usepackage{indentfirst}
\usepackage[all]{xy}
\usepackage{cite}

\theoremstyle{plain}
\newtheorem{prop} {Proposition} [section]
\newtheorem{thm}[prop] {Theorem} 
\newtheorem{dfn} [prop]{Definition}

\newtheorem{cor}[prop]{Corollary}
\newtheorem{lem}[prop]{Lemma}

\theoremstyle{definition}
\newtheorem{prob*}[prop]{Problem}
\newtheorem{rem}[prop]{Remark}

\title[SPECTRAL THEOREM]%
{A spectral theorem for a non-Archimedean valued field whose residue field is formally real}
\author[ISHIZUKA]%
{KOSUKE ISHIZUKA}
\address{Mathematical Institute, Graduate School of Science, Tohoku University, 6-3 Aramakiaza, Aoba, Sendai, Miyagi 980-8578, Japan.}
\date{}
\email{kosuke.ishizuka.r7@dc.tohoku.ac.jp}

\begin{document}
\begin{abstract}
In this paper, we will prove a spectral theorem for self-adjoint compactoid operators. Also, we will study the condition on which the coefficient field must be imposed. In order to get the theorems, we will use the Fredholm theory for compactoid operators. Moreover, the property of maximal complete field is important for our study. These facts will allow us to find that the spectral theorem depends only on the residue class field, and is independent of the valuation group of the coefficient field. As a result, we can settle the problem of the spectral theorem in the case where the residue class field is formally real.
\end{abstract}

\maketitle

\section{Introduction and preliminaries}

\subsection{Introduction}
The spectral theory on non-Archimedean functional analysis has been studied by many researchers. In this paper, we will prove the spectral theorem of self adjoint compactoid operators in the case where the residue class field is formally real (Theorem \ref{a3}, Theorem \ref{a6}, Corollary \ref{a7}). This claim was proposed in \cite[Theorem 4.3]{sp}, but the proof makes mistakes and the claim must be modified. We will give a correct proof and the exact condition in section \ref{a}. \par
For the study of the spectral theorem of compactoid operators, the Fredholm theory of compactoid operators (see Schickhof \cite{comp}) will play an important role. He found that if the coefficient field is algebraically closed, a compactoid operator is a spectral operator (\cite[Defnition 6.5]{comp}). \par
In \cite[section 6]{comp}, the coefficient field is assumed to be algebraically closed, but the assumption seems too strong for some results. Therefore, we will modify his theory to remove the assumption that the coefficient field is algebraically closed. In section \ref{app}, we summarize the discussion as an appendix. \par
As a result, we can apply the method of operator analysis to the spectral theory if the coefficient field $K$ satisfies the condition $(H)_K$ (see section \ref{a}), which is the condition on the diagonalization of a symmetric matrix. \par
In section \ref{b}, we will study the condition $(H)_K$. Keller and Ochsenius found that a symmetric matrix over $\mathbb{R}((t))$ can be diagonalized by an orthogonal matrix (see \cite{spect}). In this paper, we will extend the theorem (Theorem \ref{b2}), and get Corollary \ref{b3}. For the proof, we will use the spherical completion (c.f. \cite{van}), and the property of maximally complete field (\cite{kap}). These facts seem not to be often used, but they are very interesting themselves.

\subsection{Preliminaries}
In this paper, $K$ is a non-archimedean non-trivially valued field which is complete under the metric induced by the valuation $|\cdot| : K \to [0,\infty)$. A unit ball of $K$ is denoted by $B_K := \{ x \in K : |x| \leq 1 \}$. We denote by $k$ the residue class field of $K$.\par
Throughout, $(E,\|\cdot\|)$ is a Banach space over $K$.  Let $a \in E$, $r > 0$, we write $B_E(a,r)$ for the closed ball with radius $r$ about $a$, that is, $B_E(a,r) := \{ x \in E : \|x-a\| \leq r \}$. For a subset $X \subseteq E$, we denote by $[X]$ the $K$-vector space generated by $X$. Let $t \in (0,1]$. A sequence $(x_n)_{1 \le n \le N} \subseteq E \setminus \{0\}$, $N \in \mathbb{N} \cup \{ \infty \}$, is said to be a $t$-orthogonal sequence if for each sequence $(\lambda_n)_{1 \le n \le N} \subseteq K$, the inequality
\begin{align*}
t \cdot \max_{1 \leq i \leq N}  \| \lambda_i x_i \| \leq \| \sum_{i=1}^{N} \lambda_i x_i \| \quad 
\end{align*}
holds. A $t$-orthogonal sequence is said to be an orthogonal sequence if $t$ = 1. A subset $A$ of $E$ is said to be a compactoid if for every $r > 0$, there exist finite elements $a_1, \cdots, a_n$ of $E$ such that $A \subseteq B_E(0,r) + B_K a_1 + \cdots + B_K a_n$. \par 
Let $(F,\|\cdot\|)$ be another Banach space, we denote by $\mathcal{L}(E,F)$ the Banach space consisting of all continuous maps from $E$ to $F$ with the usual operator norm. If $(E,\|\cdot\|) = (F,\|\cdot\|)$, we write $\mathcal{L}(E) := \mathcal{L}(E,F)$. An operator $T \in \mathcal{L}(E,F)$ is said to be a compactoid operator if $T(B_E(0,1))$ is a compactoid. For details of compactoid operators, see \cite{van,sch1,comp}. \par
For a $T \in \mathcal{L}(E)$, we define its spectrum as 
\begin{align*}
  \sigma(T) := \{ \lambda \in K : \lambda I - T \, \text{is not invertible}\},
\end{align*}
where $I \in \mathcal{L}(E)$ is the identical operator on $E$, and we write
\begin{align*}
  \sigma_p(T) := \{ \lambda \in K : \text{Ker} (\lambda I - T) \neq 0 \}
\end{align*}
for eigenvalues of $T$. Also, we set
\begin{align*}
  U_T := \{ \lambda \in K : I - \lambda T \, \text{is invertible} \},
\end{align*}
and
\begin{align*}
  D_T := \{ r \in |K| : r \neq 0, B_K(0,r) \subseteq U_T \}.
\end{align*}

Let $r \in |K|$, $r \neq 0$. A function $f : B_K(0,r) \to E$ is said to be analytic if there exists a sequence $a_0, a_1, a_2, \cdots \in E$ such that $\lim_{n \to \infty} \|a_n\| r^n = 0$, and $f$ can be represented by
  \begin{align*}
    f(\lambda) = \sum_{n = 0}^{\infty} a_n \lambda^{n} \quad (\lambda \in B_K(0,r)).
  \end{align*}

\section{Non-Archimedean inner product on $c_0$}
Let $(c_0,\|\cdot\|)$ be the Banach space of all null sequences $x = (x_n)_{n \in \mathbb{N}}$ in $K$, and $\|x\| := \sup_{n \in \mathbb{N}} |x_n|$. There exists a symmetric bilinear form $\langle \cdot, \cdot \rangle$ on $c_0$ defined by
\begin{align*}
  \langle x, y \rangle := \sum_{n \in \mathbb{N}} x_n y_n
\end{align*}
where $x = (x_n)$, $y = (y_n) \in c_0$. \par

We denote by $e_1, e_2, \cdots \in c_0$ the canonical unit vectors. Then, an operator $T \in \mathcal{L}(c_0)$ can be written as a pointwise convergent sum 
\begin{align*}
  T = \sum_{i,j} a_{i,j} \cdot (e_j' \otimes e_i)
\end{align*}
where $e_j' \otimes e_i (x) := \langle e_j, x \rangle e_i$. From the perspective of this representation, we can characterize compactoid operators. 

\begin{thm}[c.f. {\cite[Theorem 8.1.9]{sch1}}] \label{c1}
  Let $T = \sum_{i,j} a_{i,j} \cdot (e_j' \otimes e_i) \in \mathcal{L}(c_0)$. Then $T$ is a compactoid operator if and only if $\lim_{i \to \infty} \sup_j |a_{i,j}| = 0$.
\end{thm}

\begin{dfn}
  We say that $T \in \mathcal{L}(c_0)$ admits an adjoint operator $S \in \mathcal{L}(c_0)$ if for each $x,y \in c_0$, $S$ satisfies
  \begin{align*}
    \langle T(x), y \rangle = \langle x, S(y) \rangle.
  \end{align*}
  If $T$ admits an adjoint operator $S$, then, since $S$ is uniquely determined by $T$, we write $T^* := S$.
\end{dfn}

It is easy to see that $T = \sum_{i,j} a_{i,j} \cdot (e_j' \otimes e_i)$ admits an adjoint operator if and only if for each $i \in \mathbb{N}$, we have $\lim_{j} a_{i,j} = 0$. If $T$ admits an adjoint operator $T^*$, then $T^*$ can  be represented by
\begin{align*}
  T^* = \sum_{i,j} a_{j,i} \cdot (e_j' \otimes e_i),
\end{align*}
and $\|T^*\| = \|T\|$. From Theorem \ref{c1}, we have the following theorem.

\begin{thm} \label{c2}
  Let $T = \sum_{i,j} a_{i,j} \cdot (e_j' \otimes e_i) \in \mathcal{L}(c_0)$. Then, $T$ is a compactoid operator which admits an adjoin operator if and only if $\lim_{n \to \infty} \sup_{n < i,j} |a_{i,j}| = 0$.
\end{thm}

In general, the symmetric bilinear form $\langle \cdot, \cdot \rangle$ on $c_0$ does not satisfy the equality $\|x\|^2 = |\langle x, x \rangle |$. On the other hand, if the residue class field $k$ of $K$ is formally real, $\langle \cdot, \cdot \rangle$ induces the norm $\|\cdot\|$ on $c_0$ (L. Narici and E. Beckenstein \cite{inner}). 

\begin{dfn}
  A field $F$ is called formally real if for any finite subset $(a_i)_{1 \le i \le n} \subseteq F$, $\sum_{1 \le i \le n} a_i^2 = 0$ implies $a_i = 0$ for each $i$.
\end{dfn}

\begin{thm}[{\cite[Theorem 6.1]{inner}}]
  Suppose the residue class field $k$ of $K$ is formally real. Then, we have $\|x\|^2 = |\langle x, x \rangle |$ for each $x \in c_0$.
\end{thm}

From now on, in this section, we suppose that the residue class field $k$ of $K$ is \textbf{formally real}.

\begin{dfn}
  A subset $X \subseteq c_0$ is called orthonormal if for each distinct pair $x, y \in X$, we have $\langle x, y \rangle = 0$.
\end{dfn}

\begin{thm}[{\cite[Theorem 3.1]{inner}}]
  Suppose that the residue class field $k$ of $K$ is formally real. Then, an orthonormal subset $X \subseteq c_0$ is orthogonal, that is, for any finite distinct elements $x_1, x_2, \cdots, x_n \in X$, the equality
  \begin{align*}
 \max_{1 \leq i \leq n}  \| \lambda_i x_i \| = \| \sum_{i=1}^{n} \lambda_i x_i \|, \quad (\lambda_1, \lambda_2, \cdots, \lambda_n \in K)
\end{align*}
holds.
\end{thm}

By the Gram-Schmidt procedure, we have the following theorem.

\begin{thm}[{\cite[Section 7]{inner}}]
  Let $M \subseteq c_0$ be a finite-dimensional subspace. Then, there exists a basis $\{x_1, \cdots, x_n\} \subseteq M$ as a $K$-vector space such that it is an orthonormal set.
\end{thm}

\begin{dfn}
  Let $X \subseteq c_0$. We denote by $X^{\perp} := \{y \in c_0 : \langle x, y \rangle = 0 \, \text{for each} \, x \in X\}$ the normal complement of $X$. A closed subspace $M \subseteq c_0$ is called normally complemented if $M \oplus M^{\perp} = c_0$.
\end{dfn}

Even if $k$ is formally real, there exists a closed subspace $M \subseteq c_0$ which is not normally complemented (\cite[Remark 9.1]{inner}). On the other hand, if $M$ is finite-dimensional, it is normally complemented.

\begin{thm}[{\cite[Corollary 8.2]{inner}}]
  Let $M \subseteq c_0$ be a finite-dimensional subspace. Then, $M$ is normally complemented.
\end{thm}

We introduce a normal projection to characterize whether a closed subspace is normally complemented.

\begin{dfn}[{\cite[Definition 6]{hil}}]
  An operator $P \in \mathcal{L}(E)$ is called a normal projection if $P^2 = P$ and $P^* = P$. 
\end{dfn}

\begin{thm}[{\cite[Corollary 3]{hil}}]
  Let $M \subseteq c_0$ be a closed subspace. Then, $M$ is normally complemented if and only if there exists a normal projection $P$ onto $M$.
\end{thm}

\begin{thm}[{\cite[Theorem 8.1]{inner}}] \label{c3}
  Let $M \subseteq c_0$ be a finite-dimensional subspace, and $\{x_1, \cdots, x_n\} \subseteq M$ be an orthonormal basis. Then, the normal projection $P$ onto $M$ can be represented by
  \begin{align*}
    P(x) = \sum_{i=1}^n \frac{\langle x, x_i \rangle}{\langle x_i, x_i \rangle} x_i.
  \end{align*}
\end{thm}

\section{The spectral theorem} \label{a}

In this section, suppose that the residue class field $k$ of $K$ is \textbf{formally real}. We say that an operator $T \in \mathcal{L}(c_0)$ is self-adjoint if $T$ admits an adjoint operator $T^*$, and $T = T^*$. We can prove the following propositions by the classical way.

\begin{prop} \label{a2}
  Let $T \in \mathcal{L}(c_0)$ be a self-adjoint operator. Then, we have $\| T ^2\| = \|T\|^2$. In particular, the equality $\lim_{n \to \infty} \|T^n\|^{1/n} = \|T\|$ holds. 
\end{prop}
\begin{proof}
  The inequality $\| T ^2\| \le \|T\|^2$ is clear. On the other hand, we have
  \begin{align*}
    \|T\|^2 &= \sup_{\|x\| \le 1} \|T(x)\|^2 \\
    &= \sup_{\|x\| \le 1} | \langle T(x), T(x) \rangle | \\
    &= \sup_{\|x\| \le 1} | \langle T^*T(x), x \rangle | \\
    &\le \sup_{\|x\|, \|y\| \le 1} |\langle T^*T(x), y \rangle | \\
    &= \|T^*T\| = \|T^2\|.
  \end{align*}
\end{proof}

\begin{prop} \label{a4}
  Let $T \in \mathcal{L}(c_0)$ be a self-adjoint operator, and $M \subseteq c_0$ be a closed subspace that is normally complemented. Then, $T(M) \subseteq M$ if and only if $TP = PT$ where $P$ is a normal projection onto $M$. In particular, $T(M) \subseteq M$ implies $T(M^{\perp}) \subseteq M^{\perp}$
\end{prop}
\begin{proof}
  Let $P$ be a normal projection onto $M$. If $TP = PT$, then it is easy to see $T(M) \subseteq M$. Conversely, suppose $T(M) \subseteq M$. Then, for each $x \in M^{\perp}$, $y \in M$, we have $\langle y, T(x) \rangle = \langle T(y), x \rangle =0$. Since $y \in M$ is arbitrary, we obtain $T(x) \in M^{\perp}$, hence $T(M^{\perp}) \subseteq M^{\perp}$. For each $z \in c_0$, we have the trivial equality
  \begin{align*}
    PT(z) + (I - P)T(z) = T(z) = TP(z) + T(I - P)(z),
  \end{align*}
  and it follows from $T(M) \subseteq M, T(M^{\perp}) \subseteq M^{\perp}$ that $TP(z) = PT(z)$.
\end{proof}

\begin{prop} \label{a5}
  Let $T \in \mathcal{L}(c_0)$ be a self-adjoint operator, and let $\lambda_1, \lambda_2 \in \sigma_p(T)$ be distinct elements. Then, for each $x_1 \in \mathrm{Ker}(\lambda_1 I - T)$, $x_2 \in \mathrm{Ker}(\lambda_2 I - T)$, we have $\langle x_1, x_2 \rangle = 0$.
\end{prop}

For a formally real field $F$, we consider the condition $(H)_F$;

\begin{itemize}
  \item[$(H)_F$] For each $n \in \mathbb{N}$ and each symmetric matrix $A \in \mathcal{M}_n(F)$, $A$ is diagonalizable over $F$.
\end{itemize}
where $\mathcal{M}_n(F)$ is the set of all $n$-dimensional square matrices over $F$.

Before proving the main theorems, we introduce the Fredholm theory for compactoid operators proved in \cite{comp}.

\begin{prop}[{\cite[Corollary 3.3, Theorem 5.6]{comp}}] \label{a1}
  Let $T \in \mathcal{L}(E)$ be a compactoid operator. Then, we have the following: \\
  $(1)$ If $\lambda \in \sigma(T)$, $\lambda \neq 0$ then $\lambda \in \sigma_p(T)$ and $\mathrm{Ker}(\lambda I - T)$ is finite-dimensional.\\
  $(2)$ If $\lambda_1, \lambda_2, \cdots \in \sigma(T)$ are distinct, then $\lim_{n \to \infty} \lambda_n = 0$. \\
\end{prop}

Also, for the proof, we use the results of section \ref{app}. For details, see section \ref{app}.

\begin{thm} \label{a3}
  Suppose that the residue class field $k$ of $K$ is formally real, and $K$ satisfies the condition $(H)_K$. Let $T \in \mathcal{L}(c_0)$ be a self-adjoint compactoid operator. Then, we have the following:\\
  $(1)$ If $K$ is densely valued, then we have
  \begin{align*}
    \|T\| = \max_{\lambda \in \sigma_p(T)} |\lambda|.
  \end{align*}
  $(2)$ If $K$ is discretely valued, then we have 
  \begin{align*}
    \|T\| \le |\pi|^{-1} \max_{\lambda \in \sigma_p(T)} |\lambda|,
  \end{align*}
  where $\pi \in B_K$ is a generating element of a maximal ideal of $B_K$.
\end{thm}
\begin{proof}
  We write $T := \sum_{i,j} a_{i,j} \cdot (e_j' \otimes e_i)$, and let $T_n := \sum_{1 \le i,j \le n} a_{i,j} \cdot (e_j' \otimes e_i)$. Then, by Theorem \ref{c2}, we have $\lim_{n \to \infty} \|T_n - T\| = 0$. Moreover, it follows from the condition $(H)_K$ that for each $n \in \mathbb{N}$ and each $r \in D_{T_n}$, the function 
  \begin{align*}
    \lambda \mapsto (I - \lambda T_n)^{-1}
  \end{align*}
  is analytic in $B_K(0,r)$. Therefore, combining these facts with Theorem \ref{app7}, we have that $T_1, T_2, \cdots$ and $T$ satisfy the hypotheses of Theorem \ref{app2}. Hence, by Corollary \ref{app3}, we have the following: \\
  $(1)$ If $K$ is densely valued, then $\lim_{n \to \infty} \|T^n\|^{1/n} = \sup_{\lambda \in \sigma(T)} |\lambda|$. \\
  $(2)$ If $K$ is discretely valued, then $\lim_{n \to \infty} \|T^n\|^{1/n} \le |\pi|^{-1} \sup_{\lambda \in \sigma(T)} |\lambda|$. \\
  Moreover, it follows from Proposition \ref{a2} that $\lim_{n \to \infty} \|T^n\|^{1/n}$ is equal to $\|T\|$, and by Proposition \ref{a1}, we have $\sup_{\lambda \in \sigma(T)} |\lambda| = \max_{\lambda \in \sigma_p(T)} |\lambda|$. This completes the proof.
\end{proof}

\begin{thm} \label{a6}
  With the same hypotheses as those of Theorem \ref{a3}, there exist an orthonormal sequence $x_1, x_2, \cdots \in c_0$ and $(\lambda_n) \in c_0$ such that 
  \begin{align*}
    T(x) = \sum_{n=1}^\infty \lambda_n \frac{\langle x, x_n \rangle}{\langle x_n, x_n \rangle} x_n.
  \end{align*}
\end{thm}
\begin{proof}
  We may assume $T \neq 0$. Then, by Theorem \ref{a3}, we have $\sigma_p(T) \setminus \{0\} \neq \emptyset$. By Proposition \ref{a1}, there exists a decreasing sequence $(r_n)_{1 \le n \le N}$ $(N \in \mathbb{N} \cup \{\infty\})$ of positive numbers such that 
  \begin{align*}
    \{ |\lambda| : \lambda \in \sigma_p(T) \setminus \{0\}  \} = \{r_n : 1 \le n \le N\}.
  \end{align*}
Moreover, we have $\lim_{n \to \infty} r_n = 0$ if $N = \infty$. \par
 For each $n \in \mathbb{N}$, we put $\{ \lambda_{n 1}, \cdots , \lambda_{n m_n}\} = \{ \lambda \in \sigma_p(T) : |\lambda| = r_n\}$ and 
  \begin{align*}
     N_n = \sum_{1 \le l \le n} \sum_{1\le k \le m_l} \text{Ker} (\lambda_{lk} I - T).
  \end{align*}
  Then, we easily have $T(N_n) \subseteq N_n$. We shall prove the theorem in the case $N = \infty$ (If $N < \infty$, the same discussion works). By Proposition \ref{a1}, $N_n$ is finite-dimensional and therefore, it follows from Theorem \ref{c3} that there exists a normal projection $P_n$ onto $N_n$. \par
  For each $n \in \mathbb{N}$, by Proposition \ref{a4} and \ref{a5}, we have
  \begin{align*}
    \sigma_p(T) = \sigma_p(TQ_n) \cup \sigma_p(TP_n) \, \text{and} \,  \sigma_p(TQ_n) = \sigma_P(T) \setminus (\bigcup_{1 \le l \le n} \bigcup_{1\le k \le m_l} \{ \lambda _{lk} \})
  \end{align*}
  where $Q_n := I - P_n$ is a normal projection onto $N_n^{\perp}$. Since $PQ_n$ is a self-adjoint compactoid operator, we have 
  \begin{align*}
    \|TQ_n\| \le C \cdot \max_{\lambda \in \sigma_p(TQ_n)} |\lambda| \le C r_{n+1}
  \end{align*}
  by Theorem \ref{a3} where $C$ is a suitable constant independent of $n$. In particular, we obtain $\lim_{n \to \infty}\|TQ_n\| = 0$ and therefore, we have $T(x) = \lim_{n \to \infty}TP_n(x)$ for each $x \in c_0$. \par
  Finally, for each $l \in \mathbb{N}$, $1 \le k \le m_l$, let $\{ x_{lkj} : 1 \le j \le p_{lk} \}$ be an orthonormal basis of $\text{Ker} (\lambda_{lk} I - T)$. Then, by Theorem \ref{c3}, and Proposition \ref{a5}, $P_n(x)$ can be represented by
 \begin{align*} 
   P_n(x) = \sum_{1 \le l \le n} \sum_{1 \le k \le m_l} \sum_{1 \le j \le p_{lk}} \frac{\langle x, x_{lkj} \rangle}{\langle x_{lkj}, x_{lkj} \rangle} x_{lkj}.
  \end{align*} 
  Hence, we have
  \begin{align*}
    T(x) = \sum_{l=1}^{\infty} \sum_{1 \le k \le m_l} \sum_{1 \le j \le p_{lk}} \lambda_{lk} \frac{\langle x, x_{lkj} \rangle}{\langle x_{lkj}, x_{lkj} \rangle} x_{lkj},
  \end{align*}
  which completes the proof.
\end{proof}

By the above theorem, we easily have the following corollary which refines Theorem \ref{a3}.

\begin{cor} \label{a7}
  With the same hypotheses as those of Theorem \ref{a3}, if $K$ is discretely valued, then we have 
  \begin{align*}
    \|T\| = \max_{\lambda \in \sigma_p(T)} |\lambda|.
  \end{align*}
\end{cor}

\begin{rem}
  Theorem \ref{a6} is the modified result of \cite[Theorem 4.3]{sp}. The fifth step of the proof of \cite[Theorem 4.3]{sp} is wrong. Moreover, it is clear that the condition $(H)_K$ is necessary for the theorem, but no condition is imposed on $K$ in \cite[Theorem 4.3]{sp}. Similarly, the proof of \cite[Theorem 10]{lev} is wrong. On the other hand, we can apply a similar method of this paper to \cite[Theorem 10]{lev}.
\end{rem}

\section{The condition $(H)_K$} \label{b}

In this section, we study the condition $(H)_K$. For a formally real field $F$, a matrix $U \in \mathcal{M}_n(F)$, $U$ is called an orthogonal matrix if its transpose $A^*$ is equal to the inverse $A^{-1}$.

\begin{prop} \label{b1}
  Let $F$ be a formally real field. Then, $F$ satisfies the condition $(H)_F$ if and only if $F$ satisfies the condition $(H')_F$;
  \begin{itemize}
    \item[$(H')_F$] for each $n \in \mathbb{N}$ and each symmetric matrix $A \in \mathcal{M}_n(F)$, $A$ can be diagonalized by an orthogonal matrix over $F$.
  \end{itemize}
\end{prop}
\begin{proof}
  Suppose that $F$ satisfies the condition $(H)_F$. Then, for each $a, b \in F$, a symmetric matrix
  \begin{align*}
  \begin{pmatrix}
    0 & \frac{b}{4} \\
    \frac{b}{4} & a \\
  \end{pmatrix}
  \end{align*}
  is diagonalizable over $F$. Hence, we have $\sqrt{a^2 + b^2} \in F$, and by induction, for any finite subset $\{a_1, \cdots, a_n \} \subseteq F$, we have $\sqrt{a_1^2 + \cdots + a_n^2} \in F$. Let $A \in \mathcal{M}_n(F)$ be a symmetric matrix. Then, by the hypothesis, there exists a subset $\{x_1, \cdots, x_n \} \subseteq F^n$ whose linear span is equal to $F^n$ such that each $x_i$ is an eigenvector of $A$. Since $A$ is symmetric, using the Gram-Schmidt procedure, we can choose $x_1, \cdots, x_n$ satisfying that a matrix $U := (x_1, \cdots, x_n)$ is an orthogonal matrix.
\end{proof}

Let $F$ be a formally real field, and let $(\Gamma, \le)$ be a totally ordered abelian group. A subset $\{c_{\alpha. \beta}\}_{(\alpha,\beta) \in \Gamma \times \Gamma} \subseteq F^*$ indexed by $\Gamma \times \Gamma$ is called a factor set if it satisfies
\begin{itemize}
  \item $c_{0,0} = c_{0,\gamma} = c_{\gamma,0} = 1$,
  \item $c_{\alpha, \beta} = c_{\beta, \alpha}$,
  \item $c_{\alpha, \beta} c_{\alpha + \beta, \gamma} = c_{\alpha, \beta + \gamma} c_{\beta, \gamma}$
\end{itemize}
for each $\alpha, \beta, \gamma \in \Gamma$. We denote by $F((\Gamma, c_{\alpha, \beta}))$ the Hahn-field defined by a factor set $\{c_{\alpha, \beta}\}$:
\begin{gather*}
  F((\Gamma, c_{\alpha, \beta})) := \big\{f : \Gamma \to F : \text{supp}f := \{\gamma \in \Gamma : f(\gamma) \neq 0\} \, \text{is a well-ordered set} \big\}, \\
  f \cdot g (\gamma) := \sum_{\alpha + \beta = \gamma}f(\alpha)g(\beta) c_{\alpha, \beta} \quad \big(f,g \in F((\Gamma, c_{\alpha, \beta})) \big).
\end{gather*}
The Hahn-field $F((\Gamma, c_{\alpha, \beta}))$ is maximally complete with respect to a general valuation $V(f) := \min \text{supp}f$, $f \in F((\Gamma, c_{\alpha, \beta}))$ (c.f. \cite{val}). The next theorem is an extension of \cite[Theorem 1]{spect}.

\begin{thm} \label{b2}
  Put $L := F((\Gamma, c_{\alpha, \beta}))$, and suppose that $F$ satisfies the condition $(H')_F$. Then, $L$ satisfies the condition $(H')_L$.
\end{thm}
\begin{proof}
  We write $f = \sum_{\gamma}f(\gamma)t^{\gamma}$ for an element $f \in L$. Let $n \ge 2$, and let $\mathcal{A} \in \mathcal{M}_n(L)$ be a symmetric matrix. Then, $\mathcal{A}$ can be represented by
  \begin{align*}
    \mathcal{A} = \sum_{\gamma \in S} A_{\gamma}\, t^{\gamma}
  \end{align*}
  where $S \subseteq \Gamma$ is a well-ordered set, and $A_{\gamma} \in \mathcal{M}_n(F)$ is a symmetric matrix for each $\gamma \in S$. To prove the theorem, we may assume that the expansion of $\mathcal{A}$ is started from $0$;
  \begin{align*}
    \mathcal{A} = A_0 + \cdots, \quad S \subseteq \{\gamma \in \Gamma : \gamma \ge 0\},
  \end{align*}
  and $A_0$ is diagonal matrix, but not a multiple of the unit matrix $I$. Moreover, after conjugating by some permutation matrix, we may assume that there exists an $r$, $1 \le r < n$, such that $A_0$ is of the form
  \begin{gather*}
    \begin{pmatrix}
      a_{11} & \cdots & 0 \\
      \vdots & \ddots & \vdots \\
      0 & \cdots & a_{nn}
    \end{pmatrix} 
  \end{gather*}
  where
  \begin{align*}
    a_{ii} = a_{11} \ \text{for} \ 1 \le i \le r, \ \text{and} \ a_{ii} \neq a_{11} \ \text{for} \ r+1 \le i \le n.
  \end{align*}
  We shall prove that there exists an orthogonal matrix $\mathcal{U} \in \mathcal{M}_n(L)$ such that $\mathcal{U}^* \mathcal{A} \mathcal{U}$ is of the form
  \begin{align*}
    \begin{pmatrix}
      \mathcal{A}_1 & 0 \\
      0 & \mathcal{A}_2
    \end{pmatrix}
  \end{align*}
  where $\mathcal{A}_1 \in \mathcal{M}_{r}(L)$, $ \mathcal{A}_2 \in \mathcal{M}_{n-r}(L)$, then by an induction on size $n$, we complete the proof. In general, we call an $n$-square matrix $(r, n-r)$-blockdiagonal if it has the shape
  \begin{align*}
    \begin{pmatrix}
      B & 0 \\
      0 & C
    \end{pmatrix}
  \end{align*}
    where $B$ is an $r$-square matrix and $C$ is an $(n-r)$-square matrix. \par 
  Let $T := \{\gamma_1 + \cdots + \gamma_n : n \in \mathbb{N}, \gamma_1, \cdots \gamma_n \in S\}$ be the semigroup generated by $S$. Then, by \cite[Theorem 3.4]{neu}, T is a well-ordered set. By the transfinite construction, we will construct a sequence $U_0, \cdots, U_{\gamma}, \cdots \in \mathcal{M}_n$ indexed by $\gamma \in T$ such that
 \begin{align*}
   & (1) \  U_0^* \, U_0 = I, \\
   & (2) \  \sum_{\substack{\alpha + \beta = \gamma \\ \alpha, \beta \in T}} U_{\alpha}^* \, U_{\beta} c_{\alpha, \beta} = 0, \ \text{for each} \ \gamma \in T, \ \gamma > 0, \\
   & (3) \ V_{\gamma} := \sum_{\substack{\alpha + \beta + \eta = \gamma \\ \alpha, \beta, \eta \in T}} U_{\alpha}^* \, A_{\beta} U_{\eta} c_{\alpha, \beta, \eta} \ \text{is $(r,n-r)$-blockdiagonal for each} \ \gamma \in T
 \end{align*} 
 where $c_{\alpha, \beta, \eta} := c_{\alpha, \beta} c_{\alpha + \beta, \eta} = c_{\alpha, \beta + \eta} c_{\beta, \eta}$, hence $c_{\alpha, \beta, \eta} = c_{\eta, \beta, \alpha}$. Then, $\mathcal{U} := \sum_{\gamma} U_{\gamma} t^{\gamma}$ is the desired orthogonal matrix. \par
  For $\gamma = 0$, we put $U_0 = I$. Let $\delta \in T$, and suppose we have determined $U_0, \cdots, U_{\gamma}, \cdots$, $\gamma < \delta$, satisfying $(1)$-$(3)$. Consider the condition $(2)$ with $\gamma = \delta$. Since $U_0 = I$, we can rewrite this condition as
  \begin{align*}
    U_{\delta}^* + U_{\delta} + \sum_{\substack{\alpha + \beta = \delta \\ \alpha, \beta \neq \delta}} U_{\alpha}^* \, U_{\beta} c_{\alpha, \beta} = 0.
  \end{align*}
  Put 
  \begin{align*}
    S_{\delta} := \sum_{\substack{\alpha + \beta = \delta \\ \alpha, \beta \neq \delta}} U_{\alpha}^* \, U_{\beta} c_{\alpha, \beta},
  \end{align*}
  then it follows from $c_{\alpha, \beta} = c_{\beta, \alpha}$ that $S_{\delta}$ is a symmetric matrix. Hence $(2)$ holds if and only if $U_{\delta}$ is of the form
  \begin{align*}
    U_{\delta} = - \frac{1}{2}S_{\delta} + Q_{\delta}
  \end{align*}
  where $Q_{\delta}$ is any antisymmetric matrix. Therefore, the task is to choose an antisymmetric matrix $Q_{\delta}$ such that $U_{\delta} = - (1/2) S_{\delta} + Q_{\delta}$ satisfies $(3)$ with $\gamma = \delta$. \par
  Now, we can rewrite $V_{\delta}$ as
  \begin{align*}
    V_{\delta} &= U_{\delta}^* \, A_0 + A_0 U_{\delta} + \sum_{\substack{\alpha + \beta + \eta = \gamma \\ \alpha, \eta \neq 0}} U_{\alpha}^* \, A_{\beta} U_{\eta} c_{\alpha, \beta, \eta} \\
    &= -Q_{\delta}A_0 + A_0 Q_{\delta} + T_{\delta}
  \end{align*}
  where
  \begin{align*}
    T_{\delta} := - \frac{1}{2}(S_{\delta} A_0 + A_0 S_{\delta}) + \sum_{\substack{\alpha + \beta + \eta = \gamma \\ \alpha, \eta \neq 0}} U_{\alpha}^* \, A_{\beta} U_{\eta} c_{\alpha, \beta, \eta}.
  \end{align*}
  Since $S_{\delta}$ and all the $A_{\gamma}$'s are symmetric, combining $c_{\alpha, \beta, \eta} = c_{\eta, \beta, \alpha}$, it follows that $T_{\delta}$ is symmetric. Notice that $T_{\delta}$ is expressed in terms of matrices already determined. \par
  Write
  \begin{align*}
    V_{\delta} = 
    \begin{pmatrix}
      v_{11} & \cdots & 0 \\
      \vdots & \ddots & \vdots \\
      0 & \cdots & v_{nn}
    \end{pmatrix},
    Q_{\delta} =
    \begin{pmatrix}
      q_{11} & \cdots & 0 \\
      \vdots & \ddots & \vdots \\
      0 & \cdots & q_{nn}
    \end{pmatrix},
    T_{\delta} = 
    \begin{pmatrix}
      t_{11} & \cdots & 0 \\
      \vdots & \ddots & \vdots \\
      0 & \cdots & t_{nn}
    \end{pmatrix}.
    \end{align*}
    Then, we have 
    \begin{align*}
      v_{ij} = -q_{ij} a_{jj} +a_{ii}q_{ij} + t_{ij} = -q_{ij}(a_{jj} - a_{ii}) + t_{ij}
    \end{align*}
    for all $1 \le i,j \le n$. If either $1 \le i \le r < j \le n$ or $1 \le j \le r < i \le n$, then by choosing $a_{ii}$, we have $a_{ii} \neq a_{jj}$. Finally, we put 
    \begin{align*}
      q_{ij} := \left\{
      \begin{aligned}
        \frac{t_{ij}}{a_{jj}-a_{ii}} \ &; \ (\text{$1 \le i \le r < j \le n$ or $1 \le j \le r < i \le n$}) \\
        0 \ &; \ \text{otherwise}
      \end{aligned}
      \right. .
    \end{align*}
    Then, we can check that $Q_{\delta}$ is antisymmetric and $V_{\delta}$ is $(r,n-r)$-blockdiagonal. This completes the proof.
\end{proof}

By using the above theorem, we can characterize the condition for which $K$ satisfies the condition $(H)_K$. 
  
\begin{thm}
  Suppose that the residue class field $k$ of $K$ is formally real. Then, $K$ satisfies the condition $(H')_K$ if and only if $k$ satisfies the condition $(H')_k$.
\end{thm}
  \begin{proof}
    The sufficiency is easy to prove by the reduction to the residue class field. Conversely, suppose that $k$ satisfies the condition $(H')_k$. Let $L$ be an immediate extension of $K$ which is maximally complete (c.f. \cite[Theorem 4.49]{van}). Then, by the well-known result (c.f. \cite[Chapter 3, Corollary to Theorem 10]{artin}), $K$ is algebraically closed in $L$. Therefore, if $L$ satisfies the condition $(H)_{L}$, then $K$ satisfies the condition $(H)_K$, hence the condition $(H')_K$ by Proposition \ref{b1}. On the other hand, by \cite[Theorem 6]{kap}, $L$ is analytically isomorphic to the Hahn-field $k((G,c_{\alpha, \beta}))$ where $G$ is the valuation group of $K$ and $\{c_{\alpha, \beta}\} \subseteq k^*$ is a factor set. Hence, by Theorem \ref{b2}, $L$ satisfies the condition $(H')_L$, which completes the proof.
\end{proof}

By Proposition \ref{b1}, we have the next corollary. Surprisingly, the condition $(H)_K$ is independent of the valuation group of $K$.         

\begin{cor} \label{b3}
  Suppose that the residue class field $k$ of $K$ is formally real. Then, $K$ satisfies the condition $(H)_K$ if and only if $k$ satisfies the condition $(H)_k$.
\end{cor}

Combining Theorem \ref{a6} with Corollary \ref{b3}, we can say that the spectral theorem of self-adjoin compactoid operators holds if the residue class field $k$ satisfies the condition $(H)_k$. The condition is independent of the valuation group of $K$. 

\begin{thm}
  Suppose that the residue class field $k$ of $K$ is formally real, and satisfies the condition $(H)_k$. Let $T \in \mathcal{L}(c_0)$ be a self-adjoint compactoid operator. Then, there exist an orthonormal sequence $x_1, x_2, \cdots \in c_0$ and $(\lambda_n) \in c_0$ such that $\langle x_n, x_n \rangle = 1$ for each $n \in \mathbb{N}$, and
  \begin{align*}
    T(x) = \sum_{n=1}^\infty \lambda_n \langle x, x_n \rangle x_n.
  \end{align*}
\end{thm}
\begin{proof}
  From Theorem \ref{a6} and Corollary \ref{b3}, it suffices to prove that for each $x \in c_0$, $x \neq 0$, we have $\sqrt{\langle x, x \rangle} \in K$. By the proof of Proposition \ref{b1}, we have $\sqrt{a_1^2 + \cdots a_n^2} \in k$ for each finite subset $\{a_1, \cdots, a_n \} \subseteq k$. Therefore, applying Hensel's lemma, we have the claim.
\end{proof}

\section{Appendix} \label{app}

In this appendix, we summarize the results of \cite[Section 6]{comp}. In \cite[Section 6]{comp}, the coefficient field $K$ is assumed to be algebraically closed. On the other hand, in this appendix, we give no condition on $K$. Hence, it can be perhaps discretely valued.

\begin{prop}[{\cite[Proposition 6.2]{comp}}] \label{app1}
  Suppose $K$ is densely valued or the residue class field $k$ of $K$ is an infinite field. Let $r \in |K|$, $r \neq 0$, and $f : B_K(0,r) \to E$,  $f(\lambda) = \sum_{n = 0}^{\infty} a_n \lambda^{n}$ be an analytic function. Then we have
  \begin{align*}
    \sup_{\lambda \in B_K(0,r)} \|f(\lambda)\| = \max_n \|a_n\|r^n.
  \end{align*}
\end{prop}
\begin{proof}
  In the case $E = K$, the conclusion of Proposition \ref{app1} is well-known. Hence, the same proof as that of \cite[Proposition 6.2]{comp} works.
\end{proof}

\begin{cor}[{\cite[Corollary 6.3]{comp}}]
  With the same hypotheses as those of Proposition \ref{app1}, the set of analytic functions $B_K(0,r) \to K$ is uniformly closed.
\end{cor}

\begin{thm}[{\cite[Lemma 6.9]{comp}}] \label{app2}
  Let $T_1, T_2, \cdots \in \mathcal{L}(E)$, and let $T = \lim_{n \to \infty} T_n$ in the sense of the operator norm. Suppose that \\
  $(1)$ for each $n \in \mathbb{N}$ and each $r \in D_{T_n}$, $(I - \lambda T_n)^{-1}$ is analytic in $B_K(0,r)$,  \\
  $(2)$ for each $r \in D_T$, $M_r := \sup_{|\lambda| \le r} \|(I - \lambda T)^{-1}\| < \infty$, and \\
  $(3)$ $K$ is densely valued or the residue class field $k$ of $K$ is an infinite field. \\
 Then, $(I - \lambda T)^{-1}$ is analytic in $B_K(0,r)$ for each $r \in D_T$.  
\end{thm}
\begin{proof}
  We can apply the same proof as that of \cite[Lemma 6.9]{comp}.
\end{proof}

\begin{cor} \label{app3}
  With the same hypotheses as those of Theorem \ref{app2}, we have the following: \\
  $(1)$ If $K$ is densely valued, then we have
  \begin{align*}
    \lim_{n \to \infty} \|T^n\|^{1/n} = \sup_{\lambda \in \sigma(T)} |\lambda|.
  \end{align*}
  $(2)$ If $K$ is discretely valued and the residue class field $k$ of $K$ is an infinite field, then we have
  \begin{align*}
    \lim_{n \to \infty} \|T^n\|^{1/n} \le |\pi|^{-1} \sup_{\lambda \in \sigma(T)} |\lambda|
  \end{align*}
  where $\pi \in B_K$ is a generating element of a maximal ideal of $B_K$.
\end{cor}
\begin{proof}
  For a sufficiently small $r > 0$, $(I - \lambda T)^{-1}$ is of the form $\sum_n (\lambda T)^n$ in $B_K(0,r)$. Therefore, by Proposition \ref{app1} and Theorem \ref{app2}, we have $(I - \lambda T)^{-1} = \sum_n (\lambda T)^n$ in $B_K(0,r)$ for each $r \in D_T$. Hence, we derive $(1)$, $(2)$.
\end{proof}

For $x_1, \cdots, x_n \in E$, we define the volume function of $x_1, \cdots, x_n \in E$ by
\begin{align*}
  \text{Vol}(x_1, \cdots, x_n) := \prod_{i=i}^n \text{dist}(x_i, [x_j : j < i])
\end{align*}
These properties can be found in \cite[Chapter 1]{note}.

From now on, when $K$ is discretely valued, we assume that a Banach space $(E,\|\cdot\|)$ satisfies $\|E\| \subseteq |K|$.

\begin{dfn}[{\cite[Definition 6.10]{comp}}]
Let $E$ be infinite-dimensional, let $T \in \mathcal{L}(E)$. For $n \in \mathbb{N}$, we set
\begin{align*}
  \Delta_n(T) &:= \sup\{ \frac{\mathrm{Vol}(T(x_1), \cdots, T(x_n))}{\mathrm{Vol}(x_1, \cdots, x_n)} : x_1, \cdots, x_n \, \mathrm{linearly} \, \mathrm{independent}\} \\
  \Delta_{-}(T) &:= \liminf_{n \to \infty}(\Delta_n(T))^{1/n} \\
  \Delta_{+}(T) &:= \limsup_{n \to \infty}(\Delta_n(T))^{1/n}.
\end{align*}
\end{dfn}

By \cite[Corollary 1.5]{note}, if $[x_1, \cdots, x_n] = [y_1, \cdots, y_n]$, then we have 
\begin{align*}
  \text{Vol}(x_1, \cdots, x_n) = \text{Vol}(y_1, \cdots, y_n).
\end{align*}
Thus, we obtain
\begin{align*}
  \Delta_n(T) = \sup \{ \text{Vol}(T(x_1), \cdots, T(x_n)) : \|x_i\| \le 1 \,\text{for each} \, i\}
\end{align*}

\begin{prop}[{\cite[Proposition 6.11]{comp}}] \label{app5}
  Let $T \in \mathcal{L}(E)$ be a compactoid operator. Then, we have $\Delta_{+}(T) = 0$.
\end{prop}
\begin{proof}
  See the proof of \cite[Proposition 6.11]{comp}.
\end{proof}

\begin{lem}[{\cite[Lemma 6.13]{comp}}] \label{app4}
  Let $x_1, \cdots, x_n \in E$, $\|x_i\| \le 1$ for each $i$ and $0 < \varepsilon < \mathrm{Vol}(x_1, \cdots, x_n)$. If $y_1, \cdots, y_n \in E$, $\|y_i - x_i\| < \varepsilon$ for each $i$, then we have $\mathrm{Vol}(x_1, \cdots, x_n) = \mathrm{Vol}(y_1, \cdots, y_n)$.
\end{lem}
\begin{proof}
  See the proof of \cite[Lemma 6.13]{comp}.
\end{proof}

The next proposition is proved in \cite{comp}, but the proof makes a little mistake. We shall give a modified proof.

\begin{prop}[{\cite[Proposition 6.12]{comp}}] \label{app6}
  Let $T \in \mathcal{L}(E)$ be such that
  \begin{align*}
    M_s = \sup_{|\lambda| \le s} \|(I - \lambda T)^{-1}\| = \infty
  \end{align*}
  for some $s \in D_T$, then we have $\Delta_{-}(T) > 0$.
\end{prop}
\begin{proof}
  By assumption, there exists a sequence $\lambda_1, \lambda_2, \cdots \in B_K(0,s)$ satisfying that $\|(I - \lambda_n T)^{-1}\|$ tends to $\infty$. Thus, there exists a sequence $y_1, y_2, \cdots \in E$ tending to $0$ such that for
  \begin{align*}
    x_n := (I - \lambda_n T)^{-1} y_n,
  \end{align*}
we have $\inf_n \|x_n\| > 0$ and $\sup_n \|x_n\| < \infty$. It follows from the same reason of part \textrm{I} of the proof of \cite[Proposition 6.12]{comp} that $\lambda_1, \lambda_2, \cdots$ does not have a convergent subsequence (part \textrm{I} of the proof of \cite[Proposition 6.12,]{comp} is correct). Thus, by taking a suitable subsequence, we may assume $\inf_{n \neq m} |\lambda_n - \lambda_m| \neq 0$ and $\inf_n |\lambda_n| > 0$. By replacing a norm $\|\cdot\|$ with a suitable norm equivalent to $\|\cdot\|$, we may assume that $\|x_n\| = 1$ for each $n \in \mathbb{N}$. Also, it is easy to see that we may assume $\|T\| < 1$, hence $s < 1$. \par

Put $\mu_n := \lambda_n^{-1}$ for each $n$, then we have
\begin{itemize}
  \item $\|T\| < 1$,
  \item $|\lambda_n| \le s < 1, |\mu_n| \le C$ for each $n$,
  \item $0 < \rho < \inf_{n \neq m} |\lambda_n - \lambda_m|, \inf_{n \neq m} |\mu_n - \mu_m| \le 1$,
  \item $\|x_n\| = 1$ for each $n$,
  \item $\lim_{n \to \infty}(x_n - \lambda_n T(x_n)) = 0, \lim_{n \to \infty}(\mu_n x_n - T(x_n)) = 0$,
\end{itemize}
where $\rho$ and $C > 1$ are suitable constants. We claim that for each $n \in \mathbb{N}$, there exists a positive number $r(n) \le 1$ such that 
\begin{align*}
  r(n) \le \text{Vol}(x_{k+1}, x_{k+2}, \cdots, x_{k+n}) 
\end{align*}
for all but finitely many $k \in \mathbb{N}$. We prove the claim by the induction on $n$. For $n = 1$, we can take $r(1) = 1$. Suppose that $r(1), r(2), \cdots, r(n-1)$ have been determined. Then, there exists a natural number $k_0 \in \mathbb{N}$ such that for each $k \ge k_0$, we have
\begin{align*}
  r(l) \le \text{Vol}(x_{k+1}, x_{k+2}, \cdots, x_{k+l}) \, (1 \le l \le n-1),
\end{align*}
and
\begin{align*}
  \|\mu_k x_k - T(x_k)\| < \varepsilon
\end{align*}
where $0 < \varepsilon < \rho \cdot t^2$, $t := \prod_{i=1}^{n-1}r(i)$. In particular, for each $k \ge k_0$, $x_{k+1}, \cdots, x_{k+n-1}$ is $t$-orthogonal. Put $r(n) := C^{-1}\rho \cdot r(n-1) \cdot t \le 1$, and we shall prove that $r(n)$ is the desired constant. In fact, by the induction hypothesis, we have 
\begin{align*}
  & \text{Vol}(x_{m+1}, x_{m+2}, \cdots, x_{m+n}) \\
  = \, & \text{dist}(x_{m+n}, [x_{m+1}, \cdots, x_{m+n-1}]) \cdot \text{Vol}(x_{m+1}, x_{m+2}, \cdots, x_{m+n-1}) \\
  \ge \, &  r(n-1) \cdot \text{dist}(x_{m+n}, [x_{m+1}, \cdots, x_{m+n-1}])
\end{align*}
for each $m \ge k_0$. Thus, we have to show that for each choice of $\xi_1, \cdots, \xi_{n-1} \in K$,
  \begin{align*}
    y := x_{m+n} - (\xi_1 x_{m+1} + \cdots + \xi_{n-1}x_{m+n-1})
  \end{align*}
has norm $\ge C^{-1} \rho t$. Since $C^{-1} \rho t \le 1$, we may assume $1 = \|x_{m+n}\| = \|\xi_1 x_{m+1} + \cdots + \xi_{n-1}x_{m+n-1}\|$. Then, we have
\begin{align*}
  & \|\sum_{i=1}^{n-1}\xi_i(\mu_{m+n} - \mu_{m+i})x_{m+i}\| \\
  = \, & \|\mu_{m+n} \cdot (\sum_{i=1}^{n-1}\xi_ix_{m+i} - x_{m+n}) + (\mu_{m+n}x_{m+n} - T(x_{m+n}))  \\ 
  & \quad + T(x_{m+n} - \sum_{i=1}^{n-1}\xi_i x_{m+i}) + \sum_{i=1}^{n-1}\xi_i \cdot (T(x_{m+i}) - \mu_{m+i}x_{m+i})\| \\
  \le & \, C\|y\| \lor \varepsilon \lor \|y\| \lor (\varepsilon \cdot \max_{1\le i \le n-1} |\xi_i|) = C \|y\| \lor \varepsilon \lor (\varepsilon \cdot \max_{1\le i \le n-1} |\xi_i|).
\end{align*}
On the other hand, by $t$-orthogonality of $x_{m+1}, \cdots, x_{m+n-1}$, we obtain
  \begin{align*}
    1 = \|\xi_1 x_{m+1} + \cdots + \xi_{n-1}x_{m+n-1}\| \ge t \cdot \max_{1\le i \le n-1} |\xi_i|
  \end{align*}
  and
\begin{align*}
  \|\sum_{i=1}^{n-1}\xi_i(\mu_{m+n} - \mu_{m+i})x_{m+i}\| &\ge t \cdot \max_{1\le i \le n-1} |\xi_i| \cdot |\mu_{m+n} - \mu_{m+i}| \\
  & \ge t \rho \cdot  \max_{1\le i \le n-1} |\xi_i| \\
  &\ge t \rho \cdot \|\xi_1 x_{m+1} + \cdots + \xi_{n-1}x_{m+n-1}\| = t \rho.
\end{align*}
Consequently, we have
  \begin{align*}
    t \rho \le C \|y\| \lor \varepsilon t^{-1}.
  \end{align*}
  By our choice $\varepsilon < \rho t^2$, we must have 
  \begin{align*}
    C^{-1}\rho t \le \|y\|,
  \end{align*}
  which proves the claim. \par
  Finally, we prove $\Delta_{-}(T) > 0$. Let $n \in \mathbb{N}$. Choose a positive number $\varepsilon'$ with $0 < \varepsilon' < r(n)$, and choose a natural number $k_0 \in \mathbb{N}$ such that $\text{Vol}(x_{k+1}, \cdots, x_{k+n}) \ge r(n)$ and $\|x_k - \lambda_k T(x_k)\| < \varepsilon'$ for all $k \ge k_0$. By Lemma \ref{app4}, we have
  \begin{align*}
    \text{Vol}(x_{k_0+1}, \cdots, x_{k_0+n}) &= \text{Vol}(\lambda_{k_0 +1}T(x_{k_0+1}), \cdots, \lambda_{k_0 + n}T(x_{k_0+n})) \\
    &= |\lambda_{k_0 +1} \cdots \lambda_{k_0 + n}|\text{Vol}(T(x_{k_0+1}), \cdots, T(x_{k_0+n})) \\
    &\le |\lambda_{k_0 +1} \cdots \lambda_{k_0 + n}| \Delta_{n}(T) \text{Vol}(x_{k_0+1}, \cdots, x_{k_0+n}).
  \end{align*}
  Therefore, we obtain
  \begin{align*}
    \Delta_{n}(T) \ge |\lambda_{k_0 +1} \cdots \lambda_{k_0 + n}|^{-1} \ge s^{-n}.
  \end{align*}
  As a consequence, we have the desired inequality $\Delta_{-}(T) \ge s^{-1} > 0$.
\end{proof}  

Combining Proposition \ref{app5} and Proposition \ref{app6}, we obtain the following theorem.

\begin{thm} \label{app7}
  Let $T \in \mathcal{L}(E)$ be a compactoid operator. Then for each $r \in D_T$, we have $M_r = \sup_{|\lambda| \le r} \|(I - \lambda T)^{-1}\| < \infty$.
\end{thm}

\section*{Acknowledgments}
 This work was supported by JST, the establishment of university fellowships
towards the creation of science technology innovation, Grant Number JPMJFS2102.


\begin{thebibliography}{99}

  \bibitem{hil} J. Aguayo and M. Nova, Non-Archimedean Hilbert like spaces, Bull. Belg. Math. Soc. Simon Stevin 14, 787–-797
(2007).

  \bibitem{sp}  J. Aguayo and M. Nova, Compact and Self-adjoint operators on Free Banach Spaces of
countable type, Advances in non-Archimedean analysis, 1--17, Contemp. Math., 665, Amer.
Math. Soc., Providence, RI (2016).

  \bibitem{lev} J. Aguayo, M. Nova, K. Shamseddine, Characterization of compact and self-adjoint operators on free Banach spaces of countable type over the complex
Levi-Civita field, J. Math. Phys. 54 (2) (2013).

   \bibitem{artin} E. Artin, \textit{Algebraic numbers and algebraic functions}, Gordon and Breach Science Publishers, New York-London-Paris (1967).

   \bibitem{kap} I. Kaplansky, Maximal fields with valuations, Duke Math. J. vol. 9 (1942) pp. 303--321.

  \bibitem{spect} H. Keller, H. Ochsenius, A spectral theorem for matrices over fields of power
series, Ann. Math. Blaise Pascal 2 (1995), 169–-179.

  \bibitem{inner}  L. Narici, E. Beckenstein, A non-Archimedean Inner Product, Contemporary Mathematics, vol. 384 (2005), 187--202.

  \bibitem{neu} B. H. Neumann, On ordered division rings, Trans. Amer. Math. Soc. 66 (1949), 202--252.

  \bibitem{sch1} C. Perez-Garcia, W. H. Schikhof, \textit{Locally convex spaces over non-archimedean valued fields}, Cambridge Studies in Advanced Mathematics, vol. 119. Cambridge University Press, Cambridge (2010).

  \bibitem{comp} W.H. Schikhof, On p-adic compact operators, Report 8911, pp. 1–-28, Department of Mathematics, Catholic University, Nijmegen, The Netherlands (1989).

  \bibitem{val}  O. F. G. Schilling, \textit{The Theory of Valuations}, Mathematical Surveys, No. 4, American Mathematical Society, New York, N. Y. (1950).

  \bibitem{van} A. C. M. van Rooij, \textit{Non-Archimedean functional analysis}, Monographs and Textbooks in
Pure and Applied Math., vol. 51, Marcel Dekker, Inc., New York (1978).

  \bibitem{note} A. C. M. van Rooij, Notes on p-adic Banach spaces, I-V: Report 7633, 1--62, Mathematisch Instituut, Katholieke Universiteit, Nijmegen, The Netherlands (1976).

  

 
  
  
\end{thebibliography}
\end{document}